\documentclass[10pt,A4paper,fleqn]{amsart}
\usepackage{mathrsfs}
\usepackage{amsfonts}
\usepackage{txfonts}
\usepackage{amsmath}
\usepackage{amssymb}
\usepackage{amsthm}
\usepackage[pdftex]{graphicx}
\usepackage[toc,page,title,titletoc,header]{appendix}
\usepackage{geometry}
\usepackage{upgreek}
\usepackage[T1]{fontenc}
\usepackage{color}
\usepackage{hyperref}
\usepackage[all]{xy}
\usepackage[french,english]{babel}

\theoremstyle{plain}
\newtheorem{theo}{Theorem}[section]
\newtheorem*{theo*}{Theorem}
\newtheorem{prop}{Proposition}[section]

\newtheorem*{lem*}{Lemma}

\theoremstyle{definition}
\newtheorem{lem}{Lemma}[section]
\newtheorem{defi}{Definition}[section]

\newtheorem{example-condition}{Example-Condition}[section]

\theoremstyle{remark}

\newtheorem*{ack}{Acknowledgements}
\newtheorem*{rem*}{Remark}
\newtheorem{rem}{Remark}[section]

\newcommand{\R}{\mathbb{R}}
\newcommand{\T}{\mathbb{T}}
\newcommand{\Z}{\mathbb{Z}}
\newcommand{\C}{\mathbb{C}}

\begin{document}

\bibliographystyle{plain}

\abstract
In this article, we first present the Kustaanheimo-Stiefel regularization of the spatial Kepler problem in a symplectic and quaternionic approach. We then establish a set of action-angle coordinates, the so-called LCF coordinates, of the Kustaanheimo-Stiefel regularized Kepler problem, which is consequently used to obtain a conjugacy relation between the integrable approximating ``quadrupolar'' system of the lunar spatial three-body problem and its regularized counterpart. This result justifies the study of Lidov-Ziglin \cite{LidovZiglin} of the quadrupolar dynamics of the lunar spatial three-body problem near degenerate inner ellipses.
\endabstract


\title{Kustaanheimo-Stiefel Regularization and the Quadrupolar Conjugacy}

\author{Lei Zhao}
\address{Johann Bernoulli Institute for Mathematics and Computer Science, University of Groningen, The Netherlands}
\email{l.zhao@rug.nl}

\date\today
\maketitle


\section{Introduction}

\subsection{Kustaanheimo-Stiefel regularization}By the use of spinors, an elegant generalization of the Levi-Civita\footnote{{which could be called Goursat-Levi-Civita regularization, as it appears already in \cite{Goursat}}.} regularization of the planar Kepler problem to the spatial case was made in Kustaanheimo \cite{Kustaanheimo1964} and Kustaanheimo-Stiefel \cite{KS}. Later, quaternionic formulations of this regularization method are presented in Waldvogel \cite{Waldvogel} and Saha \cite{Saha}. The underlying geometry was investigated in Stiefel-Scheifele \cite{SS}. Its link with Moser's regularization \cite{Moser1970} of the spatial Kepler problem was investigated in Kummer \cite{Kummer}.

Following a formula in \cite{Mikkola}, we present a quaternionic formulation of this regularization similar to \cite{Saha} (while different from \cite{Waldvogel}), and investigate its underlying symplectic geometry in this setting. 

The Kustaaheimo-Sitefel regularized Hamiltonian always has an additional $S^{1}$-symmetry resulting from a Hamiltonian $S^{1}$-action on the symplectic manifold $T^{*} (\mathbb{H}\setminus \{0\})$, where $\mathbb{H}$ designates the space of quaternions. The symplecticity of the Kustaanheimo-Stiefel transformation is explained by the corresponding reduction procedure. We shall further show that, if we restrict this regularization procedure to the cotangent bundle of some particular 2-planes called \emph{Levi-Civita planes} of $\mathbb{H}$, then we recover the Levi-Civita regularization of the planar Kepler problem. 

Following some formulae in a course of Chenciner \cite{CoursChenciner}, Féjoz in \cite{DoubleInnerCollision}  established a set of action-angle coordinates of the planar Levi-Civita regularized Kepler problem in the regularized phase space. With the help of the Levi-Civita planes, we generalize these coordinates to the spatial case by further considering the effect of rotations. This set of \emph{LCF coordinates} is analogous to a set of canonical coordinates of Levi-Civita \cite{NuovoElementi} in the phase space.
\subsection{Quadrupolar conjugacy of the spatial lunar three-body problem}

In the perturbative study of the three-body problem, the secular systems are obtained by successive averaging over the fast, Keplerian angles of the two elliptic orbits of the non-perturbing Keplerian dynamics. These systems are not well-defined when the inner ellipse degenerates, since the Keplerian fast angles \emph{i.e.}  the mean anomalies or the mean longitudes, are not well-defined when one of the elliptic orbits degenerates as for two-body collision-ejection motions. 

Nevertheless, it is shown by Féjoz \cite{DoubleInnerCollision} that in the planar three-body problem, the secular systems, which are integrable due to the presence of the SO(2)-symmetry of rotations, extend analytically to degenerate inner ellipses. On the other hand, the inner double collisions being regularized, we may define the secular regularized system by averaging over the resulting fast angles.  Moreover, in the planar case, Féjoz showed that the dynamics of the extended secular systems and that of the secular regularized systems are closely linked: the dynamics of the secular (Levi-Civita) regularized system is orbitally conjugate to the dynamics of the extended secular system after being fully symplectically reduced by the symmetries, provided that a modification of the mass of the outermost body was made.

The integrability of the secular systems are no longer guaranteed in the spatial three-body problem. Nevertheless, in the lunar spatial three-body problem where the third body is very far from the other two, it {was} noticed by Harrington \cite{Harrington1968} that the first non-trivial term in the expansion of the secular system in powers of the semi major axes ratio (the {so-called} quadrupolar system) is somehow accidentally integrable. The fully reduced system by all the symmetries extends analytically near degenerate inner ellipses, and its dynamics is extensively studied globally first by Lidov-Ziglin \cite{LidovZiglin}. 

Notice that the relevance of the quadrupolar dynamics near degenerate inner ellipses in Lidov-Ziglin's study with the real dynamics of the lunar three-body problem requires further justification, since \emph{a priori} the secular system and the quadrupolar system are not well-defined by construction.
To make such a justification, we construct the quadrupolar regularized system by the same averaging-truncating procedure in the spatial lunar three-body problem with inner double collisions regularized by the Kustaanheimo-Stiefel regularization.  By the help of the LCF coordinates, we establish the following result in the spatial lunar three-body problem:

\begin{theo} After symplectic reduction of the Keplerian $\T^{2}$-symmetry and the SO(3)-symmetry of rotations, the dynamics of the quadrupolar system, being extended to degenerate inner ellipses, is conjugate to the quadrupolar regularized dynamics up to a constant factor, provided that a modification of the mass of the outermost body {is} made. 
\end{theo}

This result generalizes the result of Féjoz from the planar three-body problem to the spatial lunar three-body problem, and simultaneously justifies the relevance of the study of Lidov-Ziglin \cite{LidovZiglin} near degenerate inner ellipses. This is used in \cite{ZLthesis} to prove the existence of a set of positive measure of almost-collision orbits in the spatial lunar three-body problem, along which the inner bodies get arbitrarily close to each other without ever collide (a slightly different approach can be found in \cite{AlmostCollision}). 

This article is organized as the follows: We start in Section \ref{Section: Preliminary} by recall{ing} some preliminaries about quaternions, and then present the Kustaanheimo-Stiefel regularization in Section \ref{Section: KS}. The Delaunay-like LCF coordinates, together with their symplecticity {are} presented in Section~\ref{Section: CF}. Finally, we investigate the conjugacy relation between the quadrupolar and the quadrupolar regularized dynamics in Section \ref{Section: Averaging and Regularization}.

\section{Preliminaries on Quaternions}\label{Section: Preliminary}
For a quaternion $z=z_0 + z_1 i+ z_2 j+ z_3 k \in \mathbb{H} \cong \R^{4}$, we denote by $Re\{z\}$ its real part $z_{0}$, and by $Im\{z\}$ its imaginary part $z_1 i+ z_2 j+ z_3 k$. 
A quaternion $z$ of the form $z_1 i+ z_2 j+ z_3 k$ (\emph{i.e.} with vanishing real part) is called a purely imaginary quaternion, and can be naturally identified with the vector $(z_{1}, z_{2}, z_{3})$ in $\R^{3}$. This identification enables us to speak of the inner (``$\cdot$'') and vector {(``$\times$'')} products  of two purely imaginary quaternions. 
The product of two quaternions is given by 
$$
z\,w=Re\{z\} Re \{w\}-Im\{z\} \cdot Im\{w\}+ Re\{z\} Im\{w\}+Re\{w\} Im\{z\}+Im\{z\} \times Im\{w\}.
$$
The conjugation $\bar{z}$ of $z$ is defined by $\bar{z}=z_{0} - z_1 i - z_2 j - z_3 k $. The modulus $\sqrt{z \cdot \bar{z}}$ of a quaternion $z$ is denoted by $|z|$.

A quaternion-valued mapping $f_{\mathbb{H}}: \mathbb{H} \rightarrow \mathbb{H}$ is called \emph{differentiable} if it is differentiable considered as a mapping from $\R^4$ to $\R^4$. 
Its derivative  $d f_{\mathbb{H}}=\dfrac{\partial{f_{\mathbb{H}}}}{\partial z_0} d z_0 + \dfrac{\partial{f_{\mathbb{H}}}}{\partial z_1} d z_1 +\dfrac{\partial{f_{\mathbb{H}}}}{\partial z_2} d z_2  + \dfrac{\partial{f_{\mathbb{H}}}}{\partial z_3} d z_3 $ is a 1-form having values in $\R^4$. We shall consider it as quaternion-valued via the identification $\R^{4} \cong \mathbb{H}$. 

Following Sudbery \cite{AS}, we define the wedge product $\phi_{\mathbb{H}}\wedge \psi_{\mathbb{H}}$ of two quaternion-valued 1-forms  $\phi_{\mathbb{H}}, \psi_{\mathbb{H}}$as:
$$
\forall v_{\mathbb{H}}, w_{\mathbb{H}} \in \mathbb{H}, \phi_{\mathbb{H}}\wedge \psi_{\mathbb{H}}(v_{\mathbb{H}}, w_{\mathbb{H}})=\phi_{\mathbb{H}} (v_{\mathbb{H}}) \psi_{\mathbb{H}} (w_{\mathbb{H}})-\phi_{\mathbb{H}} (w_{\mathbb{H}}) \psi_{\mathbb{H}} (v_{\mathbb{H}}).
$$
We shall only deal with quaternion-valued 1-forms and their wedge products in this article.

Due to the non-commutativity of the quaternion algebra, the exterior product of two quaternion-valued 1-forms is not anti-symmetric in general. In particular, the exterior product of a quaternion-valued 1-form with itself need not to be zero: we find
$$dz \wedge dz = 2 (d z_2 \wedge d z_3) i + 2 (d z_3 \wedge d z_1) j + 2 (d z_1 \wedge d z_2) k .\footnote{Note that our formula differs from \cite[Eq. $(2.35)$]{AS} by a factor of 2.}$$

Nevertheless, it can be directly verified that the real part of the wedge product of two quaternion-valued 1-forms is symmetric and is independent of the order of the two quaternions involved. 

With these notations, the canonical symplectic form on $T^*\mathbb{H}$ can be written as  $$Re\{d \bar{y} \wedge dx\} = - Re\{d \bar{x} \wedge dy\},$$
in which $x \in \mathbb{H}$, $y \in T^{*}_{x} \mathbb{H} \cong \mathbb{H}$ are the natural coordinates on the cotangent bundle $T^*\mathbb{H}$.

Rotations in $\R^3 \cong \mathbb{IH}:=\{z \in \mathbb{H}: Re\{z\}=0\}$ can be represented by unit (\emph{i.e.} of modulus 1) quaternions in the following way: Let $\rho_{1}$ be a purely imaginary quaternion and $\rho=\cos{\dfrac{\theta_{\rho}}{2}}+Im\{\rho\}$ a unit quaternion, then $\bar{\rho} \rho_{1} \rho$ is the purely imaginary quaternion rotated from $\rho_{1}$ with rotation angle $\theta_{\rho}$ and rotation axis $Im\{\rho\}$. Unit quaternions form a group $\hbox{Spin(3)} \cong \hbox{SU(2)}$, which is diffeomorphic to $\mathbb{S}^{3}$ and  doubly covers SO(3) since two unit quaternions $\rho$ and $-\rho$ determine the same rotation. Note that instead of using special coordinates (e.g. Cayley-Klein parameters etc.) on $\hbox{SU(2)}$ or SO(3), we shall stick to the use of quaternions, which leads to rather elegant formulae.

\section{Kustaanheimo-Stiefel Regularization} \label{Section: KS}
\subsection{Levi-Civita Regularization of the Kepler problem} \label{Subsection: L-C Regularization}
Before dealing with the regularization of the spatial Kepler problem, let us first recall the Levi-Civita regularization of the planar Kepler problem. Complex numbers are enough for our present purpose in this section; quaternions will appear only from Subsection \ref{Subsection: K-S. Transform}.

The Hamiltonian function of the planar Kepler problem reads
$$T(P,Q)= \dfrac{1}{2 \mu_{0}} \|P\|^2 + \dfrac{\mu_{0} M_{0}}{\|Q\|}, \qquad (P, Q) \in \C \times \C \setminus \{0\}.$$
For fixed $f>0$, we change the time from $t$ to $\tau$ on the negative energy hypersurface 
 $$\{T(P,Q)+f=0\}.$$
In the new time variable $\tau$, the flow on this energy hypersurface is given by the Hamiltonian function $\|Q\|(T(P,Q)+f).$ 

This is a function which is not even $C^{1}$ at $\{Q=0\}$. 
To obtain a system regular at $\{Q=0\}$, we make use of the Levi-Civita transformation
\begin{align*}
  L.C.: &\quad T^*(\C \setminus \{0\})  \,\,\,  \rightarrow \quad T^*\C \\
       &\quad (z,w)   \mapsto (Q=z^{2},P=\dfrac{ w }{2 \bar{z}}).
\end{align*}
As $(z, w)$ and $(-z, -w)$ have the same image, this transformation is 2-to-1.

It is direct to verify that $L.C.$ is symplectic in the sense that
$$L.C.^{*} Re(d \bar{P} \wedge d Q)=Re (d \bar{w} \wedge d z).$$

The pull-back of $\|Q\|(T(P,Q)+f)$ by $L.C.$ is found to be of the form
$$
K(z, w)=L.C.^{*}(\|Q\|(T(P,Q)+f))=\dfrac{1}{8 \mu_{0}}|w|^2+f|z|^2-\mu_{0} M_{0},
$$
which is the Hamiltonian of two harmonic oscillators in $1:1$ resonance and its dynamics can be extended regularly to $\{z=0\} \subset \C \times \C $, corresponds to the collisions $\{Q=0\}$ of the Kepler problem $T(P,Q)$. Note that it is only on the zero-energy hypersurface (diffeomorphic\footnote{All diffeomorphisms are considered as of class $C^{1}$, unless otherwise stated.} to $\mathbb{S}^{3}$) of $K(z, w)$ that its dynamics extends the dynamics of $T(P, Q)$. The compactification of an energy surface of $T(P, Q)$ determined by this regularization procedure is thus diffeomorphic to the quotient of $\mathbb{S}^{3}$ by identifying the antipodal points, which is SO(3).

\subsection{Kustaanheimo-Stiefel Transformation}\label{Subsection: K-S. Transform}


We now consider the regularization of the spatial Kepler problem.
  
By identifying $T^{*}\mathbb{H}$ with $\mathbb{H} \times \mathbb{H}$ (the fibres in $T^{*}\mathbb{H}$ are identified with the second factor), for $z, w \in \mathbb{H}$, we may consider $BL(z,w):= Re\{\bar{z} i w\}$ as a function on $T^{*}\mathbb{H}$. The \emph{bilinear relation} (as is called in \cite{SS}) 
$$BL(z,w)=0$$ thus defines a 7-dimensional quadratic cone $\Sigma \in \mathbb{H}$.

We remove the origin from the quadratic cone $\Sigma$ to obtain a 7-dimensional hypersurface $\Sigma^{0}:= \Sigma \setminus \{(0,0)\}$ of $(T^{*} \mathbb{H}, Re\{ d \bar{w} \wedge d z\})$.
Since the quadratic cone $\Sigma$ has index 4, $\Sigma^{0}$ is diffeomorphic to $\mathbb{S}^3 \times \mathbb{S}^3 \times \R $.
The quotient $V^{0}$ of $\Sigma^{0}$ by its characteristic foliation, \emph{i.e.} the foliation by the orbits of the free circle action 
$$(z, w) \mapsto (e^{i \vartheta} z, e^{i \vartheta} w)$$
is thus a smooth manifold. By standard symplectic reduction, the symplectic form $Re\{d \bar{w} \wedge d z\}$ descends to a symplectic form $\omega_{1}$ on $V^{0}$.
 
Set $\Sigma^1=\Sigma \setminus \{z=0\}$ (diffeomorphic to $\mathbb{S}^3 \times \R^3 \times \R$). This is a dense open subset of $\Sigma^{0}$ and it is invariant under the above $S^{1}$-action. Hence, the quotient space $V^{1}$ is an open dense submanifold of $V^{0}$.  By the same symplectic reduction procedure, the symplectic form $Re\{d \bar{w} \wedge d z\}$ descends to $\omega_{1}$ on $V^{1}$. 
  
\begin{defi}
The  \emph{Kustaanheimo-Stiefel} \emph{mapping} is :
\begin{align*}
  K.S.: &T^*(\mathbb{H} \setminus \{0\})  \rightarrow \mathbb{IH} \times \mathbb{H} \\
        &(z,w) \longmapsto \left(Q=\bar{z} i z,P=\dfrac{ \bar{z} i w}{2{|z|}^2}\right).
\end{align*}
\end{defi}

        The fibers of this mapping are the circles $\{(e^{i \vartheta} z, e^{i \vartheta} w), \vartheta \in \R/ (2 \pi \Z) \}$. We call the angle $\vartheta$ \emph{Kustaanheimo-Stiefel angle}. The fibers of $K.S.$ in $\Sigma^{1}$ coincide with the leaves of the characteristic foliation of $\Sigma^1$, and $\Sigma^1$ is sent to $T^*(\mathbb{IH} \setminus \{0\})$ by this mapping. One checks that $K.S.$ is surjective, and that the tangent mapping $K.S._{*}$ at every point in $\Sigma^{1}$ is surjective and has exactly a one-dimensional kernel. Therefore by inverse function theorem, $K.S.$ induces a diffeomorphism from the quotient space $V^1$ to $T^{*}(\mathbb{IH}\setminus \{0\})$.

\begin{theo}\label{prop:1}
$K.S.$ induces a symplectomorphism from $(V^{1},\omega_1)$ to $(T^*(\mathbb{IH}\setminus \{0\}), \,Re\{d \bar{P} \wedge d Q\})$.
\end{theo}

To proof this theorem, it is enough to show that the induced diffeomorphism of $K.S.$ is symplectic, which is deduced form the following lemma:

\begin{lem}\label{prop:2}
{$(K.S.)^{*}Re\{d \bar{P} \wedge d Q\}|_{\Sigma^{1}}=Re\{d \bar{w} \wedge d z \}|_{\Sigma^{1}}$.}
\end{lem}
\begin{proof}
The relation $BL(z,w)=0$ implies 
$$\bar{z} i w = \bar{w} i z,$$
which can be equivalently reformulated as
$$z^{-1} i w = \bar{w} i \bar{z}^{-1}.$$
 By differentiating the last equality, we obtain 
 $$d (z^{-1}) i w + \dfrac{\bar{z} i d w}{ |z|^{2}} = \dfrac{d \bar{w} i z}{ |z|^{2}} + \bar{w} i d (\bar{z}^{-1}).$$
From the relation 
$$0 = d(z^{-1} z) = d(z^{-1}) z + z^{-1} d z$$
 we obtain 
 $$d (z^{-1}) = - z^{-1} (d z) z^{-1}.$$ 
Also, one checks directly that 
$$Im(d \bar{z} i \wedge d z)=0.$$

Our aim is to calculate the expression
\begin{align*}
(K.S.)^{*}Re\left\{d \bar{P} \wedge d Q\right\} 
= Re\left\{d (z^{-1}i w) \wedge d \overline{(\bar{z} i z )} \right\}
=- Re\left\{ (\dfrac{1}{2} d (z^{-1})i w + \dfrac{ \bar{z} i d w}{2 |z|^2} ) \wedge (d \bar{z} i z + \bar{z} i d z) \right\}.
\end{align*}

By using the relations we have deduced beforehand, we have
\begin{align*}
                           Re\left\{ ( \dfrac{1}{2} d (z^{-1})i w + \dfrac{ \bar{z} i d w}{2 |z|^2} ) \wedge (d \bar{z} i z )  \right\}  & =   -Re\left\{ (d \bar{z} i z ) \wedge  ( \dfrac{1}{2} d (z^{-1})i w + \dfrac{ \bar{z} i d w}{2 |z|^2} )\right\} \\
                             & =   Re\left\{\dfrac{1}{2} d \bar{z} i z \wedge z^{-1} d z z^{-1} i w - d \bar{z} i z \wedge  \dfrac {\bar{z} i dw} {2 |z|^2} \right\} \\
                             & = - Re\left\{\dfrac{1}{2}  d \bar{z} i \wedge  d z z^{-1} i w\}\right\} - Re\left\{ d \bar{z} i z \wedge  \dfrac {\bar{z} i dw} {2 |z|^2} \right\} \\
                             & = \dfrac{1}{2}Re\left\{ d \bar{z} \wedge d w \right\}.
\end{align*}

\begin{align*}
                           Re\left\{( \dfrac{1}{2} d (z^{-1})i w + \dfrac{ \bar{z} i d w}{2 |z|^2} ) \wedge (\bar{z} i d z )  \right\} 
                              &= Re\left\{ ( \dfrac{1}{2} \bar{w} i d (\bar{z}^{-1}) + \dfrac{ d \bar{w} i  z}{2 |z|^2} ) \wedge (\bar{z} i d z ) \right\}  \\
                            &= -Re\left\{ ( \dfrac{1}{2} \bar{w} i \bar{z}^{-1} d (\bar{z}) \bar{z}^{-1} \wedge \bar{z} i d z  - \dfrac{ d \bar{w} i  z}{2 |z|^2}  \wedge \bar{z} i  d z \right\}   \\                          
                             & = -\dfrac{1}{2}Re\left\{ d \bar{w} \wedge d z \right\}  \\
                             & = \dfrac{1}{2}Re\left\{ d \bar{z} \wedge d w \right\} .
\end{align*}
Therefore
\begin{align*}
     (K.S.)^{*}Re\left\{d \bar{P} \wedge d Q \right\} &=  Re\left\{(d \bar{z} i z + \bar{z} i d z) \wedge ( \dfrac{1}{2} d (z^{-1})i w + \dfrac{ \bar{z} i d w}{2 |z|^2} ) \right\} \\&=-\dfrac{1}{2}Re\left\{ d \bar{z} \wedge d w \right\} - \dfrac{1}{2}Re\left\{ d \bar{z} \wedge d w \right\}  \\&=Re\left\{ d \bar{w} \wedge d z \right\} .
\end{align*}
\end{proof}

By expression, the pull-back ${{(K.S.)}^{*} F} (z, w)$ of a function $F \in C^{2}(T^*\mathbb{IH}, \R)$ naturally extends to a function defined on $T^{*} \mathbb{H} \setminus \{(0,0)\}$. We keep the same notation for the extension.

\begin{prop} \label{prop:3}
For any $F \in C^{2}(T^*\mathbb{IH}, \R)$,  the space $\Sigma^{0}$ is invariant under the Hamiltonian flow of ${{(K.S.)}^{*} F}$. 
\end{prop}
 \begin{proof} The $S^{1}$-action $\vartheta \cdot (z, w) \mapsto (e^{i \vartheta} z, e^{i \vartheta} w)$ of the Kustaanheimo-Stiefel angle on $T^{*} \mathbb{H} \setminus \{(0,0)\}$ is generated by the vector field $(i z, i w) \in T_{(z,w)}T^{*} \mathbb{H} \setminus \{(0,0)\}$, which is exactly the Hamiltonian vector field of $-BL$. Since the function ${(K. S.)}^{*} F$ is invariant under this $S^{1}$-action, the Hamiltonian flow of ${(K. S.)}^{*} F$ and $BL$ commute, and $BL$  is a first integral of the Hamiltonian flow of ${(K. S.)}^{*} F$.
\end{proof}

A Hamiltonian system $(T^*\mathbb{IH}, Re\{d \bar{P} \wedge d Q\}, F)$ is thus transformed into a Hamiltonian system $(T^*\mathbb{H}, Re\{d \bar{w} \wedge d z\}, {(K. S.)}^{*} F)$ which keeps $\Sigma^{0}$ invariant.

\subsection{Regularization of the spatial Kepler problem} \label{K.S. reg of spatial Kep prob}

The Hamiltonian of the spatial Kepler problem with mass parameters $(\mu_{0}, M_{0})$ is  of the form
$$T(P,Q)= \dfrac{1}{2 \mu_{0}} \|P\|^2 + \dfrac{\mu_{0} M_{0}}{\|Q\|},$$
where $(P,Q) \in T^*(\R^3\setminus \{0\}) \cong T^* (\mathbb{IH} \setminus \{0\})$.  We know that all negative energy levels of $T(P, Q)$ are orbitally conjugate to each other, and the flow is singular at $\{Q=0\}$. 

For any $f>0$ fixed, we change the time from $t$ to $\tau$ such that $\|Q\| \, d \tau = d t$ on the negative energy surface 
$\{T+f=0\}.$ 
In the new time variable $\tau$, the flow on $\{T+f=0\}$ is the (restricted) Hamiltonian flow of the function $\|Q\|(T + f)$, in which the velocities {remain} bounded at the limit $\{Q=0\}$. Finally, pulling back the Hamiltonian $\|Q\|(T + f)$ by $K. S.$, we find
$$
{(K. S.)}^{*}(\|Q\|(T + f))=K(z,w)=|z|^2(T(P,Q)+f) =\dfrac{1}{8 \mu_{0}}|w|^2+f|z|^2-\mu_{0} M_{0},
$$
{which describes four harmonic oscillators in $1:1:1:1$-resonance, and is well defined on the whole $T^*\mathbb{H}$ and in particular near the codimension-4 submanifold $\{z=0\}$ corresponding to the collisions of the Kepler problem.


Note that it is only on $\Sigma^{0}$ that {on its zero-energy hypersurface,} the dynamics of $K(z,w)$ extends the dynamics of the spatial Kepler problem. 
Reducing by the $S^1$-action of the Kustaanheimo-Stiefel angle, it defines a regular system on the \emph{regularized phase space} $V^0$. 
The zero-energy submanifold of the reduced space in $V^0$ is a compactification of the given Kepler (negative) energy manifold. It is obtained by adding ``at infinity" a 2-dimensional (hence codimension-3) sphere corresponding to all possible directions of collision.}

\begin{lem}The compactification of the energy surface of the spatial Kepler problem  determined by $K.S.$ is diffeomorphic to $\mathbb{S}^2 \times \mathbb{S}^3$.
\end{lem}
\begin{proof}
{The zero-energy surface of $K(z,w)= f |z|^2 + \dfrac{1}{8 \mu_{0}}|w|^2-\mu_{0} M_{0}$
is diffeomorphic to $\mathbb{S}^7$. Its intersection with the quadratic cone $\Sigma=\{(z,w):BL (z,w)=0\}$, whose index is 4, is diffeomorphic to $\mathbb{S}^3 \times \mathbb{S}^3$.}
The group $S^{1}$ acts diagonally on the intersection by 
  $$\vartheta \cdot (x,y) = (e^{i\vartheta} x, e^{i\vartheta} y), \vartheta \in \R/2 \pi \Z, (x,y) \in \mathbb{S}^3 \times \mathbb{S}^3 \subset \mathbb{H} \times \mathbb{H}.$$
  
   In order to calculate the quotient of $\mathbb{S}^3 \times \mathbb{S}^3$ by this $S^{1}$-action, we apply the diffeomorphism $(x,y) \rightarrow (x,x^{-1} y)$ from $\mathbb{S}^3 \times \mathbb{S}^3$ to itself. The diagonal $S^{1}$-action on the source space $\mathbb{S}^3 \times \mathbb{S}^3$ induces an $S^{1}$-action on the target space
$$\vartheta \cdot (x, y)\mapsto (e^{i\vartheta} x, y), \vartheta \in \R/2 \pi \Z.$$ 

The quotient of the first factor $\mathbb{S}^3$ by the $S^1$-fibers of the Hopf mapping
$$\mathbb{H} \to \mathbb{IH} \quad x \mapsto \bar{x} i x$$
being diffeomorphic to $\mathbb{S}^{2}$, the quotient of $\mathbb{S}^{3} \times \mathbb{S}^{3}$ by the $\mathbb{S}^{1}$-action is thus diffeomorphic to $\mathbb{S}^{2} \times \mathbb{S}^{3}$. 
\end{proof}

This method of regularizing the collision of the spatial Kepler problem is called \emph{Kustaanheimo-Stiefel regularization}. 
Let us sum up these discussions by a diagram:

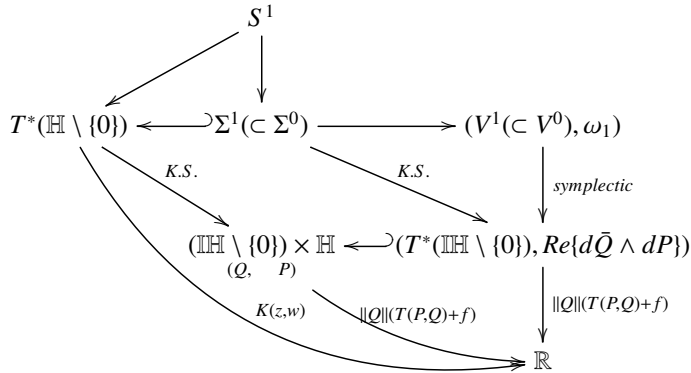
\begin{equation*}
\xymatrix@C=1.3pc@R=2pc{
  & \ar[ld] \ar[d] S^{1} &\\
T^{*} (\mathbb{H} \setminus \{0\}) \ar[dr]^{K.S.} \ar@/^-3pc/[ddrr]^{K(z,w)}& \Sigma^{1} (\subset \Sigma^{0})  \ar@{_{(}->}[l]\ar[r] \ar[dr]^{K.S.} & (V^{1} (\subset V^{0}), \omega_{1}) \ar[d]^{symplectic} \\
&\underset{(Q, \quad P)}{(\mathbb{IH} \setminus \{0\}) \times \mathbb{H}}  \ar@/^-1pc/[dr]^{\|Q\|(T(P, Q)+f)}& (T^{*} (\mathbb{IH} \setminus \{0\}), Re\{d \bar{Q} \wedge d P\})\ar[d]^{\|Q\|(T(P, Q)+f)} \ar@{_{(}->}[l]\\
& & \R\\
}
\end{equation*}

\subsection{Dynamics in the physical space}
{It is only on the regularized zero-energy level that the regularized dynamics extends the Kepler dynamics after the time has been slowed down. The regularized dynamics with non-zero regularized energy thus appears to be irrelated to the Kepler problem. Nevertheless, we notice that:}
\begin{lem} \label{EllipseCorresponding}
For any regularized energy $\tilde{f}$ satisfying $\tilde{f} > -\mu_{0} M_{0}$, the projections of the orbits of the regularized Kepler flow in the physical space are Keplerian ellipses.
\end{lem}
\begin{proof}
The equation
$$K=\dfrac{|w|^2}{8 \mu_{0}}+f|z|^2-\mu_{0} M_{0}=\tilde{f} $$
is equivalent to
 $$\|Q\|(\dfrac{\|P\|^2}{2 \mu_{0}} - \dfrac{\mu_{0} M_{0}+\tilde{f} }{\|Q\|}+f)=0,$$
 that is
 $$\dfrac{\|P\|^2}{2 \mu_{0}}-\dfrac{\mu_{0} M_{0}+\tilde{f}}{\|Q\|}=-f.$$
By assumption $\mu_{0} M_{0} + \tilde{f} >0$. The Hamiltonian flows of the left hand side in the above equality and $T(P, Q)$ are thus the same up to time parametrization. Since the orbits of the Keplerian problem with negative energy are ellipses,  the projections in the physical space of the orbits of the regularized Kepler flow are also ellipses as well. 
\end{proof}

We call these ellipses \emph{KS-ellipses}, and call the Keplerian ellipses of 
$$T(P, Q)=\dfrac{\|P\|^2}{2 \mu_{0}}-\dfrac{\mu_{0} M_{0}}{\|Q\|}$$
\emph{initial ellipses}. From Lemma \ref{EllipseCorresponding}, we see that {corresponding to the same initial condition $(P,Q)$, the KS-ellipse is just the  (initial) Keplerian ellipse, after changing the mass $M_{0}$ of the Kepler problem to $M_{0}+\tilde{f}/\mu_{0}$.}

\subsection{Levi-Civita Planes}

In $\mathbb{H}$, the orbits of the regularized Kepler problem $K(z,w)$ lie in planes spanned by two vectors $v_1, v_2 \in \mathbb{H}$ satisfying the bilinear relation
 $BL(v_1,v_2)=0.$ Such a plane is called a \emph{Levi-Civita} plane. 



\begin{lem}\label{Cor: L-C Hopf}
Suppose $x$ and $y$ are unit quaternions satisfying $BL(x,y) = 0$ and $\langle x,y \rangle = 0$, then  $\bar{x} i x = - \bar{y} i y$, and $\dfrac{1}{2}(\bar{x} i y +\bar{y} i x)= \bar{x} i y$ is a unit quaternion linearly independent of $\bar{x} i x = - \bar{y} i y$.

\end{lem}
\begin{proof}
The bilinear relation 
$$
BL(x,y)=Re\{\bar{x} i y \} = \bar{x} i y - \bar{y} i x  =0,
$$
implies
$$
 (\bar{x} y -\bar{y} x ) (\bar{x} i y - \bar{y} i x )=0,
$$
that is 
$$|y|^2 \bar{x} i x +|x|^2 \bar{y} i y =\bar{y} x \bar{y} i x + \bar{x} y \bar{x} i y,$$
$$2|y|^2 \bar{x} i x +2|x|^2 \bar{y} i y = 2(\bar{y} x + \bar{x} y)(\bar{x} i y +\bar{y} i x),$$ 
thus
$$|y|^2 \bar{x} i x +|x|^2 \bar{y} i y = \langle x,y \rangle (\bar{x} i y +\bar{y} i x)=0.$$

Since $x$ and $y$ are unit quaternions, $(\bar{x} i y +\bar{y} i x)/2= \bar{x} i y$ is also a unit quaternion. Moreover, since $x$ and $y$ are linearly independent and $\bar{x} i$ is not zero, $\bar{x} i x$ and $\bar{x} i y$ are also linearly independent.
\end{proof}

\begin{prop} 
The Hopf mapping
$$\mathbb{H} \to \mathbb{IH} \quad z \mapsto \bar{z} i z$$
 sends a Levi-Civita plane to a plane containing the origin in $\mathbb{IH}$. 
On the other hand, any plane containing the origin in $\mathbb{IH}$ is exactly the image of a $\mathbf{P}^{1}$-family of Levi-Civita planes.
\begin{proof} 
The first assertion follows from Lemma \ref{Cor: L-C Hopf}. 
To prove the second, let $\mathbf{e}_{1}, \mathbf{e}_{2}$ be an orthogonal basis of a plane in $\mathbb{IH}$. Any rotation sending $i$ to $\mathbf{e}_{1}$ determines a unit quaternion $x$ satisfying $\mathbf{e}_{1}=\bar{x} i x$. The unit quaternion $y=-i x \mathbf{e}_{2}$ thus satisfies $\mathbf{e}_{2}=\bar{x} i y$. Since $\mathbf{e}_{2}$ is purely imaginary, $$BL(x, y) =Re\{\mathbf{e}_{2}\}=0.$$ Moreover, we have $$BL(x,y)=Re\{\bar{x} y \}=- Re\{ \bar{x} i x \bar{x} i y\}=Re\{ \bar{\mathbf{e}}_{1} \mathbf{e}_{2}\}=\langle \mathbf{e}_{1}, \mathbf{e}_{2} \rangle=0.$$  The plane spanned by $x$ and $y$ is thus a Levi-Civita plane. 

Once $(x, y)$ is given, the $S^{1}$ family $\{(e^{i \vartheta} x, e^{i \vartheta} y), \vartheta \in \R/2 \pi \Z \}$ corresponds to the same $(\mathbf{e}_{1}, \mathbf{e}_{2})$, in which the pair $(e^{i \vartheta} x, e^{i \vartheta} y)$ and $(e^{i \vartheta+i\pi} x, e^{i \vartheta+i\pi} y)$ determine the same oriented Levi-Civita plane. Therefore for each oriented two-plane in $\mathbb{IH}$ passing through the origin, its pre-image is a $\mathbf{P}^{1}$-family of oriented Levi-Civita planes.

{The fibres of the Hopf map are $S^{1}$-circles each of which intersects a Levi-Civita plane in $0$ or $2$ points. Therefore the pre-image of any plane in $\mathbb{IH}$ consists exactly in a $\mathbf{P}^{1}$-family of Levi-Civita planes.}
\end{proof}
\end{prop}

An important relation between the Kustaanheimo-Stiefel transformation and the usual Levi-Civita transformation ($L.C.$) is the following:

\begin{prop}\label{prop: identification} Given a Levi-Civita plane $E$, there exist identifications to $\C$ of $E$ and of its image under the Hopf map, which make the restriction of the Hopf map into the form $$\C \to \C, \quad z \mapsto z^{2}.$$
Moreover, the restriction of $K. S.$ to $T^{*} E$ {becomes} 
$$L. C. : T^{*} \C \to T^{*} \C, \quad (z, w) \to (z^{2}, \dfrac{w}{2 \bar{z}}).$$
\end{prop}

\begin{proof}
Let $z=c_{1} \mathbf{r}_1+ c_{2} \mathbf{r}_2, w=c_{3} \mathbf{r}_1+c_{4} \mathbf{r}_2 \in E \subset \mathbb{H}$, where $c_{1}, c_{2}, c_{3}, c_{4} \in \mathbb{R}$, $\mathbf{r}_1, \mathbf{r}_2 \in E$ are two unit orthogonal quaternions. 
The Hopf map sends $z$ into
$$(c_{1}^2 -c_{2}^2)\bar{\mathbf{r}}_1 i \mathbf{r}_1 + 2 c_{1} c_{2} \bar{\mathbf{r}}_1 i \mathbf{r}_2,$$
$K.S.$ sends $(z, w)$ into
 $$\left((c_{1}^2 - c_{2}^2)\bar{\mathbf{r}}_1 i \mathbf{r}_1 + 2 c_{1} c_{2} \bar{\mathbf{r}}_1 i \mathbf{r}_2,\dfrac{( c_{1} c_{3}-  c_{2} c_{4})\bar{\mathbf{r}}_1i \mathbf{r}_1+(c_{1} c_{4} + c_{2} c_{3})\bar{\mathbf{r}}_1 i \mathbf{r}_2}{2(c_{1}^2 + c_{2}^2)}\right).$$
It is thus enough to make the identifications $\mathbf{r}_1 \sim \bar{\mathbf{r}}_1 i  \mathbf{r}_1 \sim 1$, $\mathbf{r}_2 \sim \bar{\mathbf{r}}_1 i \mathbf{r}_2 \sim i$.
\end{proof}

The restriction of the Kustaaheimo-Stiefel regularized Kepler problem to the cotangent bundle of a Levi-Civita plane is thus just the Levi-Civita regularized Kepler problem.

\section{LCF coordinates} \label{Section: CF}
\subsection{Planar case}
Let us first recall a set of Delaunay-like action-angle coordinates of the Levi-Civita regularized planar Kepler problem (that we have presented in Subsection \ref{Subsection: L-C Regularization}) obtained in \cite{DoubleInnerCollision}. In this subsection, the variables $z, w$ are complex numbers instead of being quaternions.

We first switch to the symplectic coordinates 
$$(W,Z)=\bigl(\dfrac{w}{\sqrt[4]{8 \mu_{0} f}}, \sqrt[4]{8 \mu_{0} f} z\bigr),$$
in which the function $K=K(z, w)$ is transformed into
$$K=\sqrt{\dfrac{f}{8 \mu_{1}}} (|Z|^{2}+|W|^{2})-\mu_{0} M_{0}.$$

To diagonalize the associated Hamiltonian vector field, we set
$$
(W,Z)=\bigl(\dfrac{W'+\bar{Z'}}{\sqrt{2}}, \dfrac{W'-\bar{Z'}}{i\sqrt{2}}).
$$

In $(W', Z')$ coordinates
$$
K=\sqrt{\dfrac{f}{8 \mu_{1}}} (|Z'|^{2}+|W'|^{2})-\mu_{0} M_{0},
$$
and the symplectic form is transformed to $\dfrac{i}{2} (d W' \wedge d \bar{W'} + d Z' \wedge d \bar{Z'})$.

We further switch to polar symplectic coordinates $(r_{a}, \theta_{a}, r_{b},\theta_{b})$ defined by
$$
(Z', W')=(\sqrt{2 r_{a}} e^{i \theta_{a}}, \sqrt{2 r_{b}} e^{i \theta_{b}}).
$$

In these coordinates, we have
$$K=\sqrt{\dfrac{f}{2 \mu_{1}}} (r_{a}+r_{b})-\mu_{0} M_{0}.$$

Finally, we set
$$(\mathcal{L},\delta,\mathcal{G},\gamma)=(\dfrac{r_{a}+r_{b}}{2}, \theta_{a}+\theta_{b}, \dfrac{r_{a}-r_{b}}{2}, \theta_{a}-\theta_{b}+\pi).$$
We have in these coordinates
$$K=\mathcal{L} \sqrt{\dfrac{2 f}{\mu_{0}}}-\mu_{0} M_{0},$$
and the symplectic form is transformed into the form $d \mathcal{L} \wedge d \delta + d  \mathcal{G} \wedge d \gamma$.
As has been remarked in \cite{DoubleInnerCollision}, the translation by $\pi$ in the definition of $\gamma$ is due to the reason that one considers the argument of the pericenter of an ellipse rather than its apocentre. 

\subsection{Spatial case}\label{Spatial C-F}  Following \cite{Fthesis}, we define the diffeomorphism $k_{f}$ from $V^{0}$ to itself by the following formula:
$$
k_{f}:(P,Q) \mapsto (P'=\dfrac{P}{\sqrt{2 \mu_{0} f} L},Q)
$$
so that the ellipse determined by $(P,Q)$ in the physical space under the flow of the regularized Hamiltonian $K(z,w)$ coincides with the ellipse determined by $ (P', Q)$ under $T(P,Q)$. We note that in the above formula, {$(P,Q)$ do not determine $(P', Q)$ uniquely} {since $L$, as part of the Delaunay coordinates defined below, is related to the Keplerian energy of the corresponding KS-ellipses with regularized energy $\tilde{f}$, which corresponds to a modification of the masses by $\tilde{f}$} ({Lemma \ref{EllipseCorresponding} and discussions below it}). In particular, $\sqrt{2 \mu_{0} f} L=1$ if and only if $\tilde{f}=0$.
 
{Identifying the space of spatial Keplerian ellipses of fixed semi major axis, possibly {circular or} degenerate {(to a line segment)}, to $\mathbb{S}^2\times \mathbb{S}^2$ (see \cite[Lecture 4]{Albouy}), the mapping $k_{f}$ induces the identity mapping from $\mathbb{S}^{2} \times \mathbb{S}^{2}$ to itself.} However, the mass parameters of the source and the target not being necessarily the same, the symplectic forms on the source and target space do not necessarily agree, hence the identity mapping of $\mathbb{S}^{2} \times \mathbb{S}^{2}$ is not symplectic in general. 

Let $a, e, i$ be {respectively} the semi major axis length, the eccentricity and the inclination\footnote{The use of the same $i$ for the inclination and for the imaginary unit should not cause any ambiguity for the readers.} of an elliptic {solution of the Kepler problem with Hamiltonian $T(P,Q)$}. In terms of the (symplectic) Delaunay coordinates $(L, l, G, g, H, h)$, for which
\begin{equation*} 
\left\{
\begin{array}{ll}L=\mu_0 \sqrt{M_0} \sqrt{a}   & \hbox{circular angular momentum}\\ l  &\hbox{mean anomaly}\\ G = L \sqrt{1-e^2} &\hbox{angular momentum} \\g &\hbox{argument of pericenter} \\ H=G \cos i &\hbox{vertical component of the angular momentum} \\ h &\hbox{longitude of the ascending node{,}}
\end{array}\right.
\end{equation*}
 and the diffeomorphism $k_{f}$, we define {as follows} the \emph{LCF Coordinates}, seen as coordinates on an open subset $\tilde{V}^{1}$ of the regularized phase space $V^{0}$ determined by the conditions that the corresponding KS-ellipse is non-degenerate ({in this case, even if $l$ is replaced by $u$ in this set of coordinates, the orbital plane is nevertheless not well-defined}), non-circular and non-horizontal:
\begin{equation*} 
\left\{
\begin{array}{ll}\mathcal{L}= \dfrac{\sqrt{2 f} L^{2}}{\mu_{0}^{3/2} M_{0}} \circ k_{f}  &\\ \delta=u \circ k_{f} {\hbox{, where $u$ is the eccentric anomaly}}&\\ \mathcal{G} = \dfrac{\sqrt{2 f} L G}{\mu_{0}^{3/2} M_{0}}  \circ k_{f} & \\\gamma=g \circ k_{f} & \\ \mathcal{H}=H & \\ \zeta = h.&
\end{array}
\right.
\end{equation*}
 
On the energy surface $K(z,w)=0$, we have $f=\dfrac{\mu_{0}^{3} M_{0}^{2}}{2 L^{2}}$, therefore {$k_{f}$ induces an identification of the Delaunay coordinates with the corresponding LCF coordinates} 
\begin{equation*} 
\begin{array}{ll}\mathcal{L}= L \circ k_{f}, \quad \delta=u \circ k_{f} , \quad \mathcal{G} = G  \circ k_{f} , \quad \gamma=g \circ k_{f} , \quad \mathcal{H}=H , \quad \zeta = h.
\end{array}
\end{equation*}
(except for the fast angle, which is the eccentric anomaly $u$ in LCF coordinates and {the mean anomaly} $l$ in Delaunay coordinates).

To obtain a simple proof of the symplecticity of the LCF coordinates, we shall use the following ``Rotation Lemma'':

Let $R_{1}^{I}$ be the simultaneous rotation in each factor of $\mathbb{R}^{3} \times \mathbb{R}^{3}$ around the first axis with angle $I$,  $R_{3}^{h}$ be the simultaneous rotation in each factor of $\mathbb{R}^{3} \times \mathbb{R}^{3}$ around the third (``vertical'') axis with angle $h$. Let 
$$ 
(x'_{1},x'_{2},x'_{3},y'_{1},y'_{2},y'_{3})=R_{3}^{h} \circ R_{1}^{I} (x_{1},x_{2},0,y_{1},y_{2},0).$$

\begin{lem}(Rotation Lemma) \label{Rotation Lemma} 
$$
d y'_{1} \wedge d x'_{1} + d y'_{2} \wedge d x'_{2} + d y'_{3} \wedge d x'_{3}= d y_{1} \wedge d x_{1} + d y_{2} \wedge d x_{2}+ d (x'_{1} y'_{2} - x'_{2} y'_{1}) \wedge d h.
$$
\end{lem}
\begin{proof} 

An elementary calculation leads to 
\begin{align*}
  d x'_1 \wedge d y'_1 + d x'_2 \wedge d y'_2 + d x'_3 \wedge d y'_3 
 &=d x_{1} \wedge d y_{1} + d x_{2} \wedge d y_{2} + d h \wedge d ((x_{1} y_{2}-x_{2} y_   {1}) \cdot \cos I) \\
&=d x_{1} \wedge d y_{1} + d x_{2} \wedge d y_{2} + d h \wedge d (x'_{1} y'_{2}-x'_{2} y'_{1}).
\end{align*}
\end{proof}

By Proposition \ref{prop:1}, the mapping $K.S.$ restricts to $L. C.$ on the cotangent bundle of a Levi-Civita plane. Since $(\mathcal{L},\delta,\mathcal{G},\gamma)$ form a set of Darboux coordinates, we deduce from Lemma \ref{Rotation Lemma} that

\begin{prop}\label{Delaunay-like Darboux}
LCF coordinates form a set of Darboux coordinates on $\tilde{V}^{1}$.
\end{prop}

\begin{rem} Started with a formula of \cite{CoursChenciner}, the set of coordinates $(\mathcal{L}, \delta, \mathcal{G}, \gamma)$ was established and was called ``Delaunay-like coordinates'' in \cite{DoubleInnerCollision}. On the other hand, a set of closely-related Darboux coordinates of the (non-regular) phase space has been built by Levi-Civita in \cite{NuovoElementi} both in the planar and the spatial cases with a very different approach. The nomination of LCF stands for Levi-Civita, Chenciner, and F\'ejoz.
\end{rem}

Note that LCF coordinates, similarly to the Delaunay coordinates, are not well-defined for collision-ejection Keplerian motions: there is no well-defined ``orbital plane'' for a degenerate ellipse. 

\begin{rem}As another application of Lemma \ref{Rotation Lemma}, we may deduce the symplecticity of the spatial Delaunay coordinates $(L,l,G,g,H,h)$ from that of the planar Delaunay coordinates $(L,l,G,g)$:
{indeed, from
$d y_{1} \wedge d x_{1} + d y_{2} \wedge d x_{2}=d L \wedge d l + d G \wedge d g,$ and by definition of $H=x'_{1} y'_{2}-x'_{2} y'_{1}$, one gets from Lemma \ref{Rotation Lemma} that 
$$d y_{1} \wedge d x_{1} + d y_{2} \wedge d x_{2} + d y_{3} \wedge d x_{3}=d L \wedge d l + d G \wedge d g + d H \wedge d h.$$}
\end{rem}

\section{Averaging and Regularization} \label{Section: Averaging and Regularization}
\subsection{Jacobi decomposition of the spatial three-body problem}
The spatial three-body problem is a Hamiltonian system on the phase space
$$\Pi:=\left\{(p_{j}, q_{j})_{j=0,1,2}=(p_{j}^{1}, p_{j}^{2}, p_{j}^{3}, q_{j}^{1}, q_{j}^{2}, q_{j}^{3}) \in (\R^{3} \times \R^3)^3 |\,  \forall 0 \leq j \neq k \leq 2, q_j \neq q_k \right\}, $$
with (canonical) symplectic form 
$$\sum^{2}_{j=0} \sum^{3}_{l=1} d p_j^l \wedge d q_j^l,$$
and the Hamiltonian function
$$F=\dfrac{1}{2} \sum_{0 \le j \le 2} \dfrac{\|p_j\|^2}{m_j} -  \sum_{0 \le j < k \le 2} \dfrac{m_j m_k}{\|q_j- q_k\|},$$
in which $q_0,q_1,q_2$ denote the positions of the three particles, and $p_0,p_1,p_2$ denote their conjugate momenta respectively. The Euclidean norm of a vector in $\R^{3}$ is denoted by $\|\cdot\|$. The gravitational constant has been set to $1$.

The Hamiltonian $F$ is invariant under the translations in positions. To reduce the system by this symmetry, we {switch} to the \emph{Jacobi coordinates} $(P_i, Q_i),i=0, 1, 2,$ defined as
\begin{equation*} 
\left\{
\begin{array}{l} P_0=p_0+p_1+ p_2 \\ P_1=p_1+ \sigma_1 p_2\\ P_2 = p_2
\end{array} \right.
\hbox{ \phantom{aaaaaaaqqqqaa}}
\left\{
\begin{array}{l} Q_0=q_0 \\ Q_1=q_1- q_0 \\ Q_2=q_2-\sigma_0 q_0-\sigma_1 q_1,
\end{array} \right.
\end{equation*}
where 
$$\dfrac{1}{\sigma_0}=1+\dfrac{m_1}{m_0},  \dfrac{1}{\sigma_1}=1+\dfrac{m_0}{m_1}.$$
{Due to the symmetry, the Hamiltonian function is independent of $Q_{0}$. We fix the first integral $P_{0}$ (conjugate to $Q_{0}$) at $P_{0}=0$ and reduce the translation symmetry of the system by eliminating $Q_{0}$. In coordinates $(P_i, Q_i),i=1, 2$, the (reduced) Hamiltonian function $F=F(P_{1}, Q_{1}, P_{2}, Q_{2})$ thus describes the motion of two fictitious particles. }

We further decompose the Hamiltonian $F(P_{1}, Q_{1}, P_{2}, Q_{2})$ into two parts $F=F_{Kep}+F_{pert}$, where the \emph{Keplerian part} $F_{Kep}$ and the \emph{perturbing part} $F_{pert}$ are respectively
\begin{align*}
&F_{Kep}=\dfrac{\|P_1\|^2}{2 \mu_1}+\dfrac{\|P_2\|^2}{2 \mu_2}-\dfrac{\mu_1 M_1}{\|Q_1\|}-\dfrac{\mu_2 M_2}{\|Q_2\|},
  \\&F_{pert}=-\mu_1 m_2\Bigl[\dfrac{1}{\sigma_o}\bigl(\dfrac{1}{\|Q_2-\sigma_0 Q_1\|}-\dfrac{1}{\|Q_2\|}\bigr)+\dfrac{1}{\sigma_1}\bigl(\dfrac{1}{\|Q_2+\sigma_1 Q_1\|}-\dfrac{1}{\|Q_2\|}\bigr)\,\Bigr],
 \end{align*} \label{Not: pert part}
 with (as in \cite{QuasiMotionPlanar})
 \begin{align*}&\dfrac{1}{\mu_1}=\dfrac{1}{m_0}+\dfrac{1}{m_1}, \, \dfrac{1}{\mu_2}=\dfrac{1}{m_0+m_1}+\dfrac{1}{m_2},\\ & M_1=m_0+m_1, M_2=m_0+m_1+m_2.
 \end{align*} 

We shall only be interested in the region of the phase space where $F=F_{Kep}+F_{pert}$ is a small perturbation of a pair of Keplerian elliptic motions with non-intersecting orbits. 

\subsection{Regularized Hamiltonian}
We now apply Kustaanheimo-Stiefel regularization to regularize the inner double collisions $\|Q_{1}\|=0$ of the system $F$. 

We fix a negative energy surface 
$$F=-f < 0,$$
and switch to a new time variable $\tau$ which satisfies 
$$\|Q_{1}\| \,d \tau = d t.$$ 
In time $\tau$, the corresponding motions of the particles are governed by the Hamiltonian $\|Q_{1}\| (F+f)$ and lie inside its zero-energy level. {From now on, we will call $K.S.$ the mapping}
$$(z, w, P_{2}, Q_{2}) \mapsto (Q_{1}=\bar{z} i z, P_{1}=\dfrac{\bar{z} i w}{2 |z|^{2}}, P_{2}, Q_{2})$$
and set
$$\mathcal{F}=(K.S.)^{*} \left(\|Q_{1}\|\, (F+f)\right).$$
This regularized Hamiltonian is a well-defined function on $\Sigma^{0} \times T^{*} (\R^{3} \setminus \{0\})$ and is decomposed as
$$\mathcal{F}=\mathcal{F}_{Kep}+\mathcal{F}_{pert},$$
where the \emph{regularized Keplerian part}
$$\mathcal{F}_{Kep} = K.S.^*\Bigl(\|Q_1\|(F_{Kep} + f)\Bigr)=\dfrac{|w|^2}{8 \mu_1}\, +\Bigl(f+\dfrac{\|P_2\|^2}{2 \mu_2}-\dfrac{\mu_2 M_2}{\|Q_2\|}\Bigr)|z|^2\,-\mu_1 M_1$$
describes the skew-product motion of the outer body moving on an Keplerian elliptic orbit, slowed-down by four ``inner'' harmonic oscillators in $1:1:1:1$-resonance,
and the \emph{regularized perturbing part}
$$\mathcal{F}_{pert} = K.S.^*\Bigl(\|Q_1\|F_{pert}\Bigr)$$
is small. Both terms extend analytically through the set $\{z=0\}$ corresponding to inner double collisions of $F$. 

The function $\mathcal{F}$ descends to a function (yet still called $\mathcal{F}$) on $V^{0} \setminus \{0\} \times T^{*} (\R^{3}\setminus \{0\})$ that we shall deal with in the sequel.

\subsection{The quadrupolar system and its regularized counterpart}
Due to the proper degeneracy of $F_{Kep}$ (that all bounded orbits of the Kepler problem are closed), to carry out perturbative studies of $F_{Kep}$, we must understand the dynamics the \emph{secular system} 
$$F_{sec}=\dfrac{1}{4 \pi^{2}} \int_{\T^{2}}  F_{pert} d l_{1} d l_{2},$$
in which $l_{1}, l_{2}$ are the mean anomalies of the inner and outer ellipses respectively. The system $F_{sec}$ has 6 degrees of freedom and has the Keplerian $\T^{2}$ and the rotational SO(3)-symmetry. In particular, it is not \emph{a priori} integrable.  

We now consider the lunar case of the three-body problem, that is when the ratio of the semi major axes $\alpha=\dfrac{a_{1}}{a_{2}}$ is sufficiently small, with $a_{1}$, $a_{2}$ the semi major axes of the inner and outer ellipses respectively. 
We suppose in addition that both the inner and outer ellipses are non-circular and non-degenerate so that we may use the Delaunay coordinates 
$$(L_{i}, l_{i}, G_{i}, g_{i}, H_{i}, h_{i}), \quad i=1,2$$
 to describe them respectively.

The function $F_{pert}$ is naturally an analytic function of $a_{1}, a_{2}, Q_{1}/a_{1}, Q_{2}/a_{2}$ ({by replacing $Q_{i}$ by $a_{i}\, \dfrac{Q_{i}}{a_{i}}, i=1,2$}).  With the substitution $a_{2}=\dfrac{a_{1}}{\alpha}$, it is also an analytic function of $a_{1}, \alpha, Q_{1}/a_{1}, Q_{2}/a_{2}$. By expanding $F_{sec}$ in powers of $\alpha$, we find
$$F_{sec}=F_{quad} \, \alpha^{3} + O (\alpha^{4}),$$
where, as noticed by Harrington \cite{Harrington1968} that if we perform the standard Jacobi's elimination of the nodes, that is to fix the total angular momentum vertically having norm $C$ (one has therefore $h_{1}=h_{2}+\pi$ and $H_{1}+H_{2}=C$) and to reduce by the $S^{1}$-symmetry of rotations around the vertical axis (so as to ignore the angles $h_{1}, h_{2}$ ), the (reduced) \emph{quadrupolar system}

\begin{align*}
 F_{quad}&=-\dfrac{m_{0} m_{1} m_{2}}{m_{0}+m_{1}}\dfrac{ L_{2}^{3}}{8 a_{1} G_2^3} \left\{3\dfrac{G_1^2}{L_1^2} \Bigl[1+\dfrac{(C^2-G_1^2-G_2^2)^2}{4 G_1^2 G_2^2}\Bigr] \right.\\
 &\left.\phantom{aaaaaaaaaaaa}+ 15 \Bigl(1-\dfrac{G_1^2}{L_1^2}\Bigr)\Bigl[\cos^2{g_1}+\sin^2{g_1} \dfrac{(C^2-G_1^2-G_2^2)^2}{4 G_1^2 G_2^2}\Bigr] -6\Bigl(1-\dfrac{G_1^2}{L_1^2}\Bigr)-4\right\}
\end{align*}
is integrable, since it is also independent of the outer argument of the perihelion $g_{2}$. Its dynamics of $F_{quad}$ is investigated by Lidov-Ziglin \cite{LidovZiglin} and later in Ferrer-Osacar \cite{FerrerOsacar}, Farago-Laskar \cite{FaragoLaskar} and Palaci\'an-Sayas-Yanguas \cite{PalacianSayasYanguas}.

By construction,  the system $F_{quad}$ is not defined when the inner ellipse degenerates since the angle $l_{1}$ is not defined for a degenerate inner ellipse. Nevertheless, we notice that 

\begin{prop}(Appendix \ref{Appendix: Analytic extension}) $F_{quad}$ extends analytically to degenerate inner ellipses.
\end{prop} 

The extended quadrupolar dynamics near degenerate inner ellipses was studied in Lidov-Ziglin \cite{LidovZiglin}, but it is not \emph{a priori} clear that this particular part of their study really relevant to the dynamics of the three-body problem. 

Towards a better understanding of the situation, let us first construct the secular systems of the regularized system. We suppose that the inner ellipse is non-circular, and the outer ellipse is non-circular and non-degenerate. As in \cite{AlmostCollision}, the secular regularized system $\mathcal{F}_{sec}$ is obtained analogously by averaging $\mathcal{F}_{pert}$ over the fast angles $\delta_{1}$ (the angle $\delta$ in the LCF coordinates adapted for the inner body) and $l_{2}'=l_{2} + \frac{f_{1}'(L_{2})}{2 f_{1}(L_{2})} Re\{\bar{P}_{1} Q_{1}\}$ (see \cite{DoubleInnerCollision}; the modification was made to keep the symplectic form and to have the angle $l_{2}'$ to be proportional to the new time $\tau$). We obtain from $\mathcal{F}_{sec}$ the quadrupolar regularized system $\mathcal{F}_{quad}$ by analogous expansion and truncation procedures. 

We aim to compare the (integrable) dynamics of $\mathcal{F}_{quad}$ and $F_{quad}$, in which the former is well-defined near degenerate inner ellipses by construction and the latter is not, but extends analytically to degenerate inner ellipses. In the planar three-body problem, similar situation arises and the dynamics of the regularized secular system is shown to be orbitally conjugate to the extended dynamics of the secular system \cite{DoubleInnerCollision}. Similar result can be established in our situation as well.

Let us first extend the diffeomorphism $k_f$ defined in Subsection \ref{Spatial C-F} (but keep the same notation) by only applying it to the inner body:
$$k_f :(P_1,Q_1,P_2,Q_2) \mapsto \left( \dfrac{P_1}{\sqrt{2 \mu_1 f_1(L_2)} L_{1}} , Q_1,P_2,Q_2\right).$$

\begin{prop}\label{SecularConjugacy}
The initial and secular regularized Hamiltonians satisfy:
$$\mathcal{F}_{sec}= a_1 \cdot F_{sec}  \circ k_f.$$
\end{prop}
\begin{proof}
This is a direct generalization of \cite[Proposition 3.1]{DoubleInnerCollision} to the spatial case. 
Observing that two elliptic orbits being fixed, we have $d \delta_{1} \wedge d l_{2}=d \delta_{1} \wedge d l'_{2}$ and thus
\begin{align*}
\mathcal{F}_{sec} &= \dfrac{1}{4 \pi^{2}} \int_{\T^{2}} \mathcal{F}_{pert} d \delta_{1} d l'_{2}\\
&= \dfrac{1}{4 \pi^{2}} \int_{k_{f}(\T^{2})} \mathcal{F}_{pert} \circ k_{f}^{-1} d (\delta_{1} \circ k_{f}^{-1}) d l_{2} .\\
\end{align*}
Since the map $k_{f}$ preserves the configuration coordinates $(Q_{1}, Q_{2})$ and $\mathcal{F}_{pert} = \|Q_{1}\| F_{pert}$ is only a function of the configuration variables, $\mathcal{F}_{pert} = \|Q_{1}\| F_{pert}$ is invariant under $k_{f}$. Moreover, $\delta_{1} \circ k_{f}^{-1} = u_{1}$, therefore
\begin{align*}
\mathcal{F}_{sec} &= \dfrac{1}{4 \pi^{2}} \int_{\T^{2}} \|Q_{1}\| F_{pert} d u_{1} d l_{2}\\
&= \dfrac{a_{1}}{4 \pi^{2}} \int_{\T^{2}}  F_{pert} d l_{1} d l_{2}\\
&= F_{sec}.\\
\end{align*}
The last equality follows from the relation 
$$\|Q_{1}\| \, d u_{1} = a_{1} d l_{1},$$ 
derived the Kepler equation
$$l_{1}=u_{1}-e_{1} \sin u_{1}.$$
\end{proof}

We thus obtain
$$\mathcal{F}_{quad}= a_1 \cdot F_{quad}  \circ k_f.$$

Unfortunately, the mapping $k_{f}$ is not symplectic for the secular symplectic structures involved, hence the dynamics of $\mathcal{F}_{quad}$ is not directly equivalent to that of $F_{quad}$. Nevertheless, the following theorem gives a direct and simple link between them.

\begin{theo}\label{theo:quadrupolar conjugacy}
For fixed masses $m_0$, $m_1$ and $m_2$,  semi major axes $a_1 << a_2$, energy $-f$<0, and angular momentum $C$, after full reduction by the SO(3)-symmetry and the Keplerian $\T^2$-action of the fast angles, there exists a fictitious value $m'_2 > 0$ of the outer mass, such that up to a factor depending only on $a_{1}$ and the masses $m_{2}$ and $m'_{2}$, the system $\mathcal{F}_{quad}$ is conjugated to $F_{quad}$, provided that $m'_2$ substitutes for $m_2$ in $F_{quad}$. 
\end{theo}

\begin{proof} 
We fix $\vec{C}$ vertical, so that by standard Jacobi's elimination of the nodes, $(\mathcal{G}_{1}, \gamma_{1}, G_{2}, g_{2})$ form a set of Darboux coordinates on the reduced source space of $k_{f}$, and $(G_{1}, g_{1}, G_{2}, g_{2})$ form a set of Darboux coordinates on the reduced target space of $k_{f}$.

Consider the system
$$F=F_{Kep}+F_{pert}= - \dfrac{\mu_{1}^{3} M_{1}^{2}}{2 L_{1}^{2}}-\dfrac{\mu_{2}^{3} M_{2}^{2}}{2 L_{2}^{2}}+F_{pert}$$
on the energy level 
$$F=-f, \quad f >0.$$
Since for small enough  $\alpha$,  $|F_{pert}|$ is smaller than $\dfrac{\mu_{2}^{3} M_{2}^{2}}{2 L_{2}^{2}}$, we have that 
$$\dfrac{\mu_{1}^{3} M_{1}^{2}}{2 L_{1}^{2}}<f.$$

Now as the mapping 
$$m_2 \mapsto \mu_2^{3} M_2^{2}=\dfrac{(m_{0}+m_{1})^{3} m_{2}^{3}}{m_{0}+m_{1}+m_{2}}$$
 is a diffeomorphism from $(0,+\infty)$ to itself for any positive $m_{0}, m_{1}$, there exists some $m'_2>0$, such that  
$$f_{1} (L_{2},m_{0},m_{1},m'_{2})=f-\dfrac{\mu_{2}^{3} M_{2}^{2}}{2 L_{2}^{2}}=\dfrac{\mu_{1}^{3} M_{1}^{2}}{2 L_{1}^{2}}.$$
 The composition of the mapping $k_f$ with the mapping $m: m_2 \mapsto m'_2$ is thus just the identity between $\mathcal{L}_{1}$ and $L_{1}$, and between two Darboux charts $(\mathcal{G}_{1}, \gamma_{1}, G_{2}, g_{2})$ and $(G_{1}, g_{1}, G_{2}, g_{2})$ in the reduced source and target spaces of spatial Keplerian ellipses\footnote{Remind that circular inner or outer orbits or degenerate outer orbits are removed.} respectively. Moreover, since these charts are open and dense in the source and target spaces, we conclude that this composition identifies the reduced source and the target space symplectically. 
 
We deduce from $\mathcal{F}_{quad}= a_1 \cdot F_{quad}  \circ k_f$ and the expression of $F_{quad}$ that 
$$\dfrac{m_{2}'}{m_{2}} \mathcal{F}_{quad}=\mathcal{F}_{quad} \circ m=a_{1} \cdot F_{quad} \circ k_{f} \circ m.$$
Hence the dynamical behaviors of the reduced quadrupolar regularized system agrees, up to a factor $\dfrac{a_{1} m_{2}}{m_{2}'}$, with that of the non-regularized reduced quadrupolar system. 
\end{proof}

Therefore the study of Lidov-Ziglin \cite{LidovZiglin} of $F_{quad}$ near degenerate inner ellipses is indeed relevant to the regularized secular dynamics of the spatial lunar three-body problem in the way manifested by Theorem \ref{theo:quadrupolar conjugacy}. In \cite{ZLthesis}, This is used in an essential way to establish the existence of almost-collision orbits in the spatial lunar three-body problem.

\begin{rem} The necessity of shifting one of the masses in the above theorem is due to the {possibly non-zero} energy of the regularized system. Instead of modifying the outer mass $m_{2}$, another possibility is to modify the masses $m_{0}$ and $m_{1}$. This is explored in \cite{AlmostCollision}.
\end{rem}

\appendix 
\section{Analytic extension of $F_{quad}$ to degenerate inner ellipses}\label{Appendix: Analytic extension}
In this appendix, we show by direct calculation that 

\begin{prop} The function $F_{quad}$ extends to an analytic function in a neighborhood of degenerate inner ellipses.
\end{prop} 

\begin{proof}The space of (inner) spatial Keplerian ellipse is homeomorphic to $\mathbb{S}^{2} \times \mathbb{S}^{2}$ (see e.g. \cite[Lecture 4]{Albouy}). The Pauli-Souriau coordinates for the inner spatial Keplerian ellipse $(A_{1}, A_{2}, A_{3}, B_{1}, B_{2}, B_{3})$ with $$A_{1}^{2}+A_{2}^{2}+A_{3}^{2}=B_{1}^{2}+B_{2}^{2}+B_{3}^{2}=L_{1}$$ are just the Descartes coordinates for a couple of points on $\mathbb{S}^{2}_{L_{1}}\subset \R^{3}$ (we set the radius of the sphere $\mathbb{S}^{2}_{L_{1}}$ to be $\sqrt{L_{1}}$). The angular momentum of the inner ellipse is then $\vec{C}_{1}=(\frac{A_{1}-B_{1}}{2}, \frac{A_{2}-B_{2}}{2}, \frac{A_{3}-B_{3}}{3}),$
the direction of the inner pericenter is the direction of $(-\frac{A_{1}+B_{1}}{2}, -\frac{A_{2}+B_{2}}{2}, -\frac{A_{3}+B_{3}}{3})$, and the direction of the inner ascending node is the direction of $(\frac{B_{2}-A_{2}}{2}, \frac{A_{1}-B_{1}}{2}, 0)$.

A calculation leads to
$$F_{quad}=\dfrac{\mu_{1} m_{2}}{2 \alpha^{3}} \int_{\T^{2}} \dfrac{\|Q_{1}\|^{2}}{\|Q_{2}\|^{3}} (3 \cos^{2} \zeta-1) d l_{1} d l_{2}.$$

We take the Laplace plane, \emph{the plane orthogonal to $\vec{C}$} to be the reference plane. In terms of $(e_{1}, g_{1}, i_{1}, e_{2}, i_{2})$ (for which let us restrict $i_{i}$ to the interval $[0, \pi)$\, ), this function takes the form
\small
\begin{align*}
F_{quad}&=-\dfrac{\mu_{1} m_{2}}{8 a_{1} (1-e_{2}^{2})^{\frac{3}{2}}} [3(1-e_{1}^{2})(1+\cos^{2} (i_{1}-i_{2}))+15(\cos^{2}g_{1}+\cos^{2}(i_{1}-i_{2}) \sin^{2} g_{1})-6 e_{1}^{2}-4]\\
              &=-\dfrac{\mu_{1} m_{2}}{8 a_{1} (1-e_{2}^{2})^{\frac{3}{2}}} [-(3(1-e_{1}^{2})+15 \sin^{2} g_{1}) \sin^{2} (i_{1}-i_{2}) + 12 (1-e_{1}^{2})+5].
\end{align*}
\normalsize

We see from this expression that the analyticity of (the extension of) $F_{quad}$ near degenerate inner ellipses directly follows from the analyticity of (the extensions of) the expressions $(1-e_{1}^{2}) \sin^{2} (i_{1}-i_{2})$ and $\sin^{2} g_{1} \sin^{2} (i_{1}-i_{2})$ near degenerate inner ellipses. In Pauli-Souriau coordinates and in terms of the normal vector of the outer ellipse $\vec{N}_{2}=(N_{1}, N_{2}, N_{3})$, these expressions can be written in the following form:
\small
\begin{align*}
(1-e_{1}^{2}) \sin^{2} (i_{1}-i_{2})&
          = \dfrac{1}{4L_{1}^{2}}(A_{1}-B_{1})^{2}+(A_{2}-B_{2})^{2}+(A_{3}-B_{3})^{2}\\
&\quad-\Bigl((A_{1}-B_{1}) N_{1}+(A_{2}-B_{2}) N_{2}+(A_{3}-B_{3}) N_{3}\Bigr)^{2},\\
\sin^{2} g_{1} \sin^{2} (i_{1}-i_{2}) &= \dfrac{((A_{1}+B_{1}) N_{1}+(A_{2}+B_{2}) N_{2}+(A_{3}+B_{3}) N_{3})^{2}}{(A_{1}+B_{1})^{2}+(A_{2}+B_{2})^{2}+(A_{3}+B_{3})^{2}}.
\end{align*}
\normalsize

These formulae show that they can be extended analytically to the set  $$\{(A_{1}, A_{2}, A_{3})=(B_{1}, B_{2}, B_{3})\},$$ corresponding to degenerate inner ellipses. This shows that $F_{quad}$ can be extended analytically to degenerate inner ellipses.

\end{proof}

\begin{rem}
We provide a geometrical way to calculate the expression $$\sin^{2} g_{1} \sin^{2} (i_{1}-i_{2}).$$ Take any vector $\vec{p}$ in the direction of the inner pericenter and its projection $\vec{p}_{1}$ in the outer orbital plane. Let $\vec{p}_{2}$ be the projection of $\vec{p}$ to the direction of node. It is direct to verify that the direction of the node is perpendicular to $\vec{p}_{1}-\vec{p}_{2}$. We have $$\sin^{2} g_{1}=\dfrac{\|\vec{p}-\vec{p}_{2}\|^{2}}{\|\vec{p}\|^{2}}$$ and $$\sin^{2} (i_{1}-i_{2})=\dfrac{\|\vec{p}-\vec{p}_{1}\|^{2}}{\|\vec{p}-\vec{p}_{2}\|^{2}},$$ therefore $$\sin^{2} g_{1} \sin^{2} (i_{1}-i_{2})=\dfrac{\|\vec{p}-\vec{p}_{1}\|^{2}}{\|\vec{p}\|^{2}}.$$ This is a quantity which only depends on the direction of the inner pericenter and the normal direction of the outer orbital plane, while both directions are well defined up to degenerate inner ellipses.
\end{rem}

\begin{ack}These results are part of my Ph.D. thesis \cite{ZLthesis} prepared at the Paris Observatory and the Paris Diderot University. Many thanks to my supervisors Alain Chenciner and Jacques Féjoz. It is due to their help that this note could ever exist and takes its actual form. Sincere thanks to the anonymous referee for pertinent critics and helpful suggestions.
\end{ack}

\bibliography{QuasiperiodicMotionSpatial}
\end{document}